\documentclass[12pt]{amsart}
\usepackage{amscd,amssymb,latexsym}
\usepackage{fullpage}
\newcommand{\Mdef}[2]{\newcommand{#1}{\relax \ifmmode #2 \else $#2$\fi}}

\newcommand{\finbuilds}{\models}

\newcommand{\builds}{\vdash}
\newcommand{\sm }{\wedge}

\newcommand{\tensor}{\otimes}

\newcommand{\map}{\mathrm{map}}
\newcommand{\Hom}{\mathrm{Hom}}

\newcommand{\Ext}{\mathrm{Ext}}
\Mdef{\bhom}{\mathbf{\hat{H}om}}

\Mdef{\Mod}{\mathrm{mod}}

\newtheorem{thm}{Theorem}[section]
\newtheorem{lemma}[thm]{Lemma}
\newtheorem{prop}[thm]{Proposition}
\newtheorem{cor}[thm]{Corollary}

\theoremstyle{definition}

\newtheorem{defn}[thm]{Definition}

\newtheorem{assumption}[thm]{Assumption}

\newtheorem{example}[thm]{Example}

\newtheorem{remark}[thm]{Remark}

\newcommand{\qqed}{\qed \\[1ex]}
\renewenvironment{proof}[1][\hspace*{-.8ex}]{\noindent {\bf Proof #1:\;}}{\qqed}

\Mdef{\PH} {\Phi^H}
\Mdef{\PK} {\Phi^K}
\Mdef{\PL} {\Phi^L}
\Mdef{\PT} {\Phi^{\T}}

\Mdef{\ef}{E{\cF}_+}
\Mdef{\etf}{\widetilde{E}{\cF}}
\Mdef{\eg}{E{G}_+}
\Mdef{\etg}{\tilde{E}{G}}

\Mdef{\infl}{\mathrm{inf}}
\Mdef{\defl}{\mathrm{def}}
\Mdef{\res}{\mathrm{res}}
\Mdef{\ind}{\mathrm{ind}}
\Mdef{\coind}{\mathrm{coind}}

\Mdef{\univ}{\mathcal{U}}

\Mdef{\Fp}{\mathbb{F}_p}
\Mdef{\Zpinfty}{\Z /p^{\infty}}
\Mdef{\Zpadic}{\Z_p^{\wedge}}

\newcommand{\bi}{\begin{itemize}}
\newcommand{\be}{\begin{enumerate}}
\newcommand{\bc}{\begin{center}}
\newcommand{\bd}{\begin{description}}
\newcommand{\ei}{\end{itemize}}
\newcommand{\ee}{\end{enumerate}}
\newcommand{\ec}{\end{center}}
\newcommand{\ed}{\end{description}}

\newcommand{\lra}{\longrightarrow}
\newcommand{\lla}{\longleftarrow}

\Mdef{\we}{\mathbf{we}}
\Mdef{\fib}{\mathbf{fib}}
\Mdef{\cof}{\mathbf{cof}}
\Mdef{\BI}{\mathcal{BI}}

\newcommand{\fibre}{\mathrm{fibre}}

\newcommand{\ilim}{\mathop{ \mathop{\mathrm{lim}} \limits_\leftarrow} \nolimits}

\newcommand{\holim}{\mathop{ \mathop{\mathrm {holim}} \limits_\leftarrow} \nolimits}

\Mdef{\A}{\mathbb{A}}
\Mdef{\B}{\mathbb{B}}
\Mdef{\C}{\mathbb{C}}
\Mdef{\D}{\mathbb{D}}
\Mdef{\E}{\mathbb{E}}
\Mdef{\T}{\mathbb{T}}
\Mdef{\F}{\mathbb{F}}
\Mdef{\G}{\mathbb{G}}
\Mdef{\I}{\mathbb{I}}
\Mdef{\N}{\mathbb{N}}
\Mdef{\Q}{\mathbb{Q}}
\Mdef{\R}{\mathbb{R}}
\Mdef{\bbS}{\mathbb{S}}
\Mdef{\Z}{\mathbb{Z}}

\Mdef{\bA}{\mathbb{A}}
\Mdef{\bB}{\mathbb{B}}
\Mdef{\bC}{\mathbb{C}}
\Mdef{\bD}{\mathbb{D}}
\Mdef{\bE}{\mathbb{E}}
\Mdef{\bF}{\mathbb{F}}
\Mdef{\bG}{\mathbb{G}}
\Mdef{\bH}{\mathbb{H}}
\Mdef{\bI}{\mathbb{I}}
\Mdef{\bJ}{\mathbb{J}}
\Mdef{\bK}{\mathbb{K}}
\Mdef{\bL}{\mathbb{L}}
\Mdef{\bM}{\mathbb{M}}
\Mdef{\bN}{\mathbb{N}}
\Mdef{\bO}{\mathbb{O}}
\Mdef{\bP}{\mathbb{P}}
\Mdef{\bQ}{\mathbb{Q}}
\Mdef{\bR}{\mathbb{R}}
\Mdef{\bS}{\mathbb{S}}
\Mdef{\bT}{\mathbb{T}}
\Mdef{\bU}{\mathbb{U}}
\Mdef{\bV}{\mathbb{V}}
\Mdef{\bW}{\mathbb{W}}
\Mdef{\bX}{\mathbb{X}}
\Mdef{\bY}{\mathbb{Y}}
\Mdef{\bZ}{\mathbb{Z}}

\Mdef{\cA}{\mathcal{A}}
\Mdef{\cB}{\mathcal{B}}
\Mdef{\cC}{\mathcal{C}}
\Mdef{\mcD}{\mathcal{D}} 
\Mdef{\cE}{\mathcal{E}}
\Mdef{\cF}{\mathcal{F}}
\Mdef{\cG}{\mathcal{G}}
\Mdef{\mcH}{\mathcal{H}} 
\Mdef{\cI}{\mathcal{I}}
\Mdef{\cJ}{\mathcal{J}}
\Mdef{\cK}{\mathcal{K}}
\Mdef{\mcL}{\mathcal{L}}

\Mdef{\cM}{\mathcal{M}}
\Mdef{\cN}{\mathcal{N}}
\Mdef{\cO}{\mathcal{O}}
\Mdef{\cP}{\mathcal{P}}
\Mdef{\cQ}{\mathcal{Q}}
\Mdef{\mcR}{\mathcal{R}}
\Mdef{\cS}{\mathcal{S}}
\Mdef{\cT}{\mathcal{T}}
\Mdef{\cU}{\mathcal{U}}
\Mdef{\cV}{\mathcal{V}}
\Mdef{\cW}{\mathcal{W}}
\Mdef{\cX}{\mathcal{X}}
\Mdef{\cY}{\mathcal{Y}}
\Mdef{\cZ}{\mathcal{Z}}

\Mdef{\ca}{\mathcal{a}}
\Mdef{\ct}{\mathcal{t}}

\Mdef{\At}{\tilde{A}}
\Mdef{\Bt}{\tilde{B}}
\Mdef{\Ct}{\tilde{C}}
\Mdef{\Et}{\tilde{E}}
\Mdef{\Ht}{\tilde{H}}
\Mdef{\Kt}{\tilde{K}}
\Mdef{\Lt}{\tilde{L}}
\Mdef{\Mt}{\tilde{M}}
\Mdef{\Nt}{\tilde{N}}
\Mdef{\Pt}{\tilde{P}}

\Mdef{\tA}{\tilde{A}}
\Mdef{\tB}{\tilde{B}}
\Mdef{\tC}{\tilde{C}}
\Mdef{\tE}{\tilde{E}}
\Mdef{\tH}{\tilde{H}}
\Mdef{\tK}{\tilde{K}}
\Mdef{\tL}{\tilde{L}}
\Mdef{\tM}{\tilde{M}}
\Mdef{\tN}{\tilde{N}}
\Mdef{\tP}{\tilde{P}}

\Mdef{\ft}{\tilde{f}}
\Mdef{\xt}{\tilde{x}}
\Mdef{\yt}{\tilde{y}}

\Mdef{\Ab}{\overline{A}}
\Mdef{\Bb}{\overline{B}}
\Mdef{\Cb}{\overline{C}}
\Mdef{\Db}{\overline{D}}
\Mdef{\Eb}{\overline{E}}
\Mdef{\Fb}{\overline{F}}
\Mdef{\Gb}{\overline{G}}
\Mdef{\Hb}{\overline{H}}
\Mdef{\Ib}{\overline{I}}
\Mdef{\Jb}{\overline{J}}
\Mdef{\Kb}{\overline{K}}
\Mdef{\Lb}{\overline{L}}
\Mdef{\Mb}{\overline{M}}
\Mdef{\Nb}{\overline{N}}
\Mdef{\Ob}{\overline{O}}
\Mdef{\Pb}{\overline{P}}
\Mdef{\Qb}{\overline{Q}}
\Mdef{\Rb}{\overline{R}}
\Mdef{\Sb}{\overline{S}}
\Mdef{\Tb}{\overline{T}}
\Mdef{\Ub}{\overline{U}}
\Mdef{\Vb}{\overline{V}}
\Mdef{\Wb}{\overline{W}}
\Mdef{\Xb}{\overline{X}}
\Mdef{\Yb}{\overline{Y}}
\Mdef{\Zb}{\overline{Z}}

\Mdef{\db}{\overline{d}}
\Mdef{\hb}{\overline{h}}
\Mdef{\qb}{\overline{q}}
\Mdef{\rb}{\overline{r}}
\Mdef{\tb}{\overline{t}}
\Mdef{\ub}{\overline{u}}
\Mdef{\vb}{\overline{v}}

\Mdef{\hc}{\hat{c}}
\Mdef{\he}{\hat{e}}
\Mdef{\hf}{\hat{f}}
\Mdef{\hA}{\hat{A}}
\Mdef{\hH}{\hat{H}}
\Mdef{\hJ}{\hat{J}}
\Mdef{\hM}{\hat{M}}
\Mdef{\hP}{\hat{P}}
\Mdef{\hQ}{\hat{Q}}

\Mdef{\thetab}{\overline{\theta}}
\Mdef{\phib}{\overline{\phi}}

\Mdef{\uA}{\underline{A}}
\Mdef{\uB}{\underline{B}}
\Mdef{\uC}{\underline{C}}
\Mdef{\uD}{\underline{D}}

\Mdef{\bolda}{\mathbf{a}}
\Mdef{\boldb}{\mathbf{b}}
\Mdef{\bfD}{\mathbf{D}}

\Mdef{\fm}{\frak{m}}
\Mdef{\fp}{\frak{p}}

\Mdef{\eps}{\epsilon}

\input{xypic}

\newcommand{\cell}{\mathrm{cell}}
\newcommand{\shift}{\mathrm{shift}}
\newcommand{\Ftwo}{\mathbb{F}_2}

\newcommand{\pifinite}{$\pi_*$-finite}
\newcommand{\Ca}{C}
\newcommand{\Ba}{B}

\newcommand{\TbC}{\Tb C}
\newcommand{\TbB}{\Tb B}

\begin{document}
\title{Ausoni-B\"okstedt duality for topological Hochschild homology}
\author{J.P.C.Greenlees}
\address{School of Mathematics and Statistics, Hicks Building, 
Sheffield S3 7RH. UK.}
\email{j.greenlees@sheffield.ac.uk}
\date{}

\begin{abstract}
We consider the Gorenstein condition for topological Hochschild
homology, and show that it holds remarkably often. More precisely,
if $R$ is a commutative ring spectrum and $R\lra k$ is a map to a 
field of characteristic $p$ then, provided $k$ is small as an
$R$-module,  $THH(R;k)$ is Gorenstein
in the sense of \cite{DGI1}. In particular, this holds if  $R$ is a
(conventional) regular local ring with residue field $k$ of
characteristic $p$. 

Using only B\"okstedt's calculation of 
$THH(k)$, this gives a non-calculational proof of dualities observed
by  B\"okstedt \cite{Bok} and Ausoni \cite{Ausoni},
Lindenstrauss-Madsen \cite{LM},   Angeltweit-Rognes \cite{AR}
and others.    
\end{abstract}

\maketitle

\tableofcontents


\section{Introduction}

The present work was stimulated by calculations of topological
Hochschild homology by B\"okstedt \cite{Bok} and Ausoni
\cite{Ausoni}.  Given a map of commutative ring spectra $R\lra k$, we may view
$k$ as an $R$-bimodule and hence define $THH_*(R;k)$. Identifying
conventional rings with Eilenberg-MacLane ring spectra, we may take
$k$ to be a field, and the calculations give striking examples where  
$THH_*(R;k)$ is Gorenstein. These and other calculations are summarized in  Section \ref{sec:egs},
and the reader unfamiliar with them may wish to glance at them before
proceeding. 

The purpose of the present paper is to give a non-calculational
explanation of this duality. Most are instances of the following result
which covers many cases where there is currently no complete
calculation. The basic definitions of $THH$ are described in Section 
\ref{sec:HH}  and the Gorenstein apparatus is described in Section 
\ref{sec:Gor}. The following theorem appears below as Corollary
\ref{cor:gor}.

\begin{thm} 
If $R$ is a connective commutative ring spectrum with a map $R\lra k$ where $k$ is
a field of characteristic $p>0$ then provided (i) $k$ is small as a
$R$-module and (ii) $R$ is Gorenstein of shift $a$ then $THH(R; k)$ 
is Gorenstein of shift $-a-3$ and has Noetherian homotopy groups. 
\end{thm}

This should be contrasted with the algebraic case (i.e., working under
$k$ rather than under the sphere spectrum). In this case, if $R$ is a
$k$-algebra, under the same hypotheses we expect $HH_*(R|k;k)$ to 
be Gorenstein a shift of $-a$ (see Remark \ref{rem:overk} and 
Appendix \ref{sec:NAR1}).

It is essential to the argument that we are working in characteristic
$p$,  and the  only calculational input is B\"okstedt's result that  
$THH_*(\Fp)=\Fp [\mu_2]$ (this is Gorenstein of shift  $-3$, which
explains the $-3$  in the statement of the theorem). The two
technical ingredients are  (A) a cofibre sequence
conjectured to explain the Gorenstein calculations and proved by
Dundas and (B) an extension of the usual Gorenstein
ascent  theorem.

There is a strong precedent for calculations based on Gorenstein ascent.  Indeed
if $S\lra R \lra Q$ is a cofibre sequence (i.e., $Q\simeq
R\tensor_Sk$)\footnote{These are often called {\em fibre} sequences in
the algebra literature because of the fact that cochains on topological fibre
sequences give examples, but this would lead to confusion in the
present context.},  the Gorenstein property often behaves well 
in the sense that if $S$ and $Q$ are Gorenstein then so is $R$. 
To illustrate its use we show how this, together with Morita
invariance,  lets us generate most Gorenstein rings from from
0-dimensional ones.  To start with,  exterior algebras $E$ are Poincar\'e
duality algebras (and hence 0-dimensional Gorenstein rings) and since any 
polynomial algebra $P$ is Morita equivalent to an exterior algebra
$E$, polynomial algebras are Gorenstein. Next, by Noether normalization 
any Noetherian $k$-algebra  $R$ is finitely generated as a module over a polynomial
subring $P$  so that the cofibre $Q=R\tensor_Pk$ is finite 
dimensional. Since $P$ is Gorenstein, $R$ is Gorenstein if and only 
if $Q$ is a Poincar\'e duality algebra, so Gorenstein rings are constructed from a
polynomial algebra and a Poincar\'e duality algebra. Altogether,
Gorenstein rings are constructed from  Poincar\'e duality algebras using
Morita equivalences and cofibre sequences.  

Duality phenomena are also ubiquitous in topology, starting with 
Poincar\'e duality and moving on to  coefficient rings of many equivariant
cohomology theories \cite{DGI1}. Once again it seems these all come
from a rather small collection of basic examples, namely the chains on
a group or the cochains on a manifold. Using Morita equivalences and
cofibre sequences one can generate a wide variety of further examples,
perhaps most notably $C^*(BG)$ for compact Lie
groups $G$ whose adjoint representation is orientable \cite{DGI1}. The present
paper shows that B\"okstedt's calculation provides a new source of
Gorenstein examples. \\[2ex]

The rest of the paper is organized as follows. Sections \ref{sec:HH}, \ref{sec:SS} and \ref{sec:Gor} provide
summaries of relevant background. Section \ref{sec:egs} gives
summaries of various calculations from the literature. Section
\ref{sec:Gorascent} proves the Gorenstein Ascent result we need, and
 Section \ref{sec:ABduality} proves the main result. We finish in
 Section \ref{sec:egs2} by discussing the  implications of the result
 for a number of  examples. There are two appendices which 
describe similar results. In Appendix A we consider   THH of Thom spectra via the 
work of Blumberg-Cohen-Schlichtkrull \cite{BCS}, and in Appendix B
we consider algebraic Hochschild homology (i.e., under $k$  rather
than under the sphere spectrum) where a result of Dwyer-Miller gives
analogous duality statements.

During the genesis of the paper, I have discussed my speculations with
many people, and I am grateful to V.Angeltveit, C.Ausoni, D.Benson,
B.Dundas,  W.G.Dwyer and A.Lindenstrauss for their patience and sharing their ideas. I
 am particularly grateful to B.Dundas for providing a proof of the
 critical conjectured cofibration described in Lemma
 \ref{lem:cofibresequence}, and allowing me to publish it here. 
I would like to thank the University of Lille for inviting me to give
lectures on duality in  2012, when these ideas started to make progress, and 
MSRI for providing an excellent environment for completing this
account. 

\section{Hochschild homology and cohomology}
\label{sec:HH}
We suppose given maps $S\lra R\lra k$ of ring spectra; as usual we
include the case of conventional rings through the use of
Eilenberg-MacLane spectra. 
We write $R^e=R^e_S=R\tensor_SR$, and we write
$P$ for an $R^e$-module, which we refer to as an $(R,R)$-bimodule over
$S$. Thus we may talk of the Hochschild homology spectrum
$$HH_{\bullet}(R|S;P)=R\tensor_{R^e} P$$
and the Hochschild cohomology spectrum 
$$HH^{\bullet}(R|S;P)=\Hom_{R^e}(R, P). $$
The Hochschild homology and cohomology groups are obtained by taking
homotopy groups
$$HH_*(R|S;P)=\pi_*HH_{\bullet}(R|S;P) \mbox{ and }
HH^*(R|S;P)=\pi_*HH^{\bullet}(R|S;P). $$
If $R$ and $S$ are conventional rings, $R$ is flat over $S$ and $P$ is
a conventional module, this agrees with the standard definitions in
algebra. 

In the examples of most concern to us here, 
$S=\bS$ is the sphere spectrum, and we have 
we have topological Hochschild homology $THH(R;P):=HH_{\bullet}(R|\bS ;P)$.

\section{Two spectral sequences}
\label{sec:SS}
We are going to be concerned with cofibre sequences of commutative
ring spectra, $S\lra R \lra Q$, in the sense that $Q\simeq R\tensor_S k$.
In the topological context one is used to having spectral sequences as
basic calculational tools. It is convenient to have them available
more generally. 

\subsection{The connective case}

This is the situation when the ring spectra $S, R$ and $Q$ are all connective.
The first example of this is analogous to what happens when we
have a short exact sequence of compact Lie groups $1\lra
N \lra G \lra G/N\lra 1$. We take $S=C_*(N), R=C_*(G)$ and $Q=C_*(G/N)$ and in
this case we have the homological Serre spectral sequence
$$E^2_{*,*}=H_*(G/N; H_*(N))\Rightarrow H_*(G)$$
 of the fibration $N\lra G \lra G/N$.

\begin{lemma}
\label{lem:connSS}
If  $S\lra R \lra Q$ is a cofibre sequence of connective
commutative algebras augmented over $k$ and $\pi_0(S)=k$, and $R$ is 
upward finite type as an $S$ module (for example
\cite[3.13]{DGI1}  if $\pi_n(R)$ is
finite dimensional for each $n$)
 then there is a multiplicative spectral sequence
$$E^2_{s,t}=\pi_{s}(Q)\tensor_k \pi_{t}(S)\Rightarrow
\pi_{s+t}(R),  $$
with differentials
$$d^r: E^r_{s,t}\lra E^r_{s-r,t+r-1}.$$
\end{lemma}

\begin{proof}
We construct a tower 
$$S\lra \cdots \lra S^{(n)}\lra
S^{(n-1)} \lra \cdots \lra S^{(1)}\lra S^{(0)}=k$$
of commutative rings by killing homotopy groups, where $S\lra S^{(n)}$ is an isomorphism of
$\pi_{\leq n}$ and $\pi_i(S^{(n)})=0 $ for $i>n$. Taking $R^{(n)}=R\tensor_SS^{(n)}$ we
obtain a corresponding multiplicative filtration of $R$, and consider the resulting
spectral sequence. In view of the cofibre sequence $\Sigma^t
\pi_tS\lra S^{(t)}\lra S^{(t-1)}$ of $S$-modules it is easy to write
down an exact couple, and we grade it so that  
$$D^{2}_{s,t}=\pi_{s+t}(R\tensor_S S^{(t)}) \mbox{ and }  E^{2}_{s,t}=\pi_{s+t}(R\tensor_S \Sigma^t\pi_t(S)).  $$
Now the action of $S$ on $\pi_tS$ factors through $S^{(0)}=k$, and
hence
$$R\tensor_S \Sigma^t\pi_t(S))\simeq \Sigma^t\pi_t(S)\tensor_k Q.   $$
The $d^2$ differential is induced by 
$$R\tensor_S \Sigma^t\pi_t(S)\lra R\tensor_S S^{(t)}\lra
R\tensor_S \Sigma^{t+2}\pi_{t+1}(S).    $$

The convergence of the spectral sequence is the statement that the
natural map 
$$\kappa: R=R\tensor_S S\simeq R\tensor_S [\holim_nS^{(n)}]\lra \holim_n
[R\tensor_S S^{(n)}]$$
is an equivalence. Since the $S^{(n)}$ are uniformly bounded below (by
$-1$) the result follows (as summarized in Lemma \ref{lem:limtensor} below).
\end{proof}

\subsection{The coconnective case}
This is the situation when the ring spectra $S, R$ and $Q$ are all
coconnective. This does not play a role in our main applications and
is included for comparison. Because free commutative algebras are not
usually coconnective we will need to add a significant hypothesis. 

One example of this arises if we start from a fibration $F\lra E
\lra B$  with $B$ simply connected and take  $S=C^*(B)$, $R=C^*(E)$
and $Q=C^*(F)$. This  obviously satisfies the stringent additional
hypothesis identified below, so the construction generalizes the Serre spectral 
sequence
$$E_2^{*,*}=H^*(B; H^*(F))\Rightarrow H^*(E). $$

\begin{lemma}
\label{lem:coconnSS}
Suppose 
$S\lra R \lra Q$ is a cofibre sequence of coconnective
commutative algebras augmented over $k$ and $\pi_0(S)=k$ and $R$ is 
downward finite type as an $S$ module (for example \cite[3.14]{DGI1} if $\pi_*(S)$ is
simply coconnected (i.e., $\pi_0(S)=k$ and $\pi_{-1}(S)=0$) and
$\pi_n(R)$ is finite dimensional for each $n$).
If in addition that there is a tower
$$S\lra \cdots \lra S_{(n)}\lra
S_{(n-1)} \lra \cdots \lra S_{(1)}\lra S_{(0)}=k$$
of  coconnective commutative rings with $\pi_i(S_{(n)})=0$ for $i<n$.
 then there is a multiplicative spectral sequence
$$E_2^{s,t}=\pi_{-s}(S)\tensor_k \pi_{-t}(Q)\Rightarrow
\pi_{-s-t}(R),  $$
with differentials
$$d_r: E_r^{s,t}\lra E_r^{s-r,t+r-1}.$$
\end{lemma}

\begin{proof}
By hypothesis, there is a tower
$$S\lra \cdots \lra S_{(n)}\lra 
S_{(n-1)} \lra \cdots \lra S_{(1)}\lra S_{(0)}=k$$ 
of commutative rings. 
Thus the map  $S\lra S_{(n)}$ 
is an isomorphism of $\pi_{\geq -n}$ and we have cofibre sequences of $S$-modules
$$\Sigma^{-s}\pi_{-s}S\lra S_{(s)}\lra S_{(s-1)}.$$

We then get a spectral sequence
$$D_2^{s,t}=\pi_{-s-t}(R\tensor_S S_{(s)})
\mbox{ and } E_2^{s,t}=
\pi_{-s-t}(R\tensor_S \Sigma^{-s}\pi_{-s}S)$$
The differentials then take the familiar cohomological form
$$d_r: E_r^{s,t}\lra E_r^{s+r, t-r+1}. $$

The convergence of the spectral sequence is the statement that the
natural map 
$$R=R\tensor_S S\simeq R\tensor_S \holim_nS_{(n)}\lra \holim_n
R\tensor_S S_{(n)}$$
is an equivalence, and by Lemma \ref{lem:limtensor} below this holds if $R$ is of downward finite type
over $S$. 
\end{proof}

\section{Gorenstein ring spectra}
\label{sec:Gor}
We are considering duality phenomena modelled on those in commutative
algebra of Noetherian rings, namely those associated to Gorenstein
local rings. For ring spectra there is a corresponding development, starting
by restricting the class of rings by a finiteness condition and then
the core Gorenstein condition followed by a duality statement. 
 We recall some definitions from \cite{DGI1}. 

\subsection{Finiteness conditions}

In a triangulated category if $N$ can be {\em finitely built} from $M$ using cofibre
sequences, finite sums and retracts,  we write $M\finbuilds N$; if $N$ 
can be {\em built} from $M$ using cofibre sequences and arbitrary sums we 
write $M\builds N$.

We consider a map $R\lra k$ of rings. The terminology comes from
 the special case when $R$ is a commutative local ring with residue 
field $k$.  The first requirement is a finiteness condition, which
plays the role of the Noetherian condition from classical commutative
algebra. The Auslander-Buchsbaum-Serre theorem in commutative algebra
states that if $R$ is a Noetherian local ring, $k$ is small if and only if $R$ is
regular. This is far too strong a condition for us to assume,  but
there is a much weaker and more practical one in the same
vein. Indeed, for commutative Noetherian local rings,  we can always form the Koszul
complex $K$ associated to a finite set of generators of the maximal
ideal; this has the properties (i) $K$ is small ($R\finbuilds K$) (ii) $K$ is finitely
built from $k$ ($k\finbuilds K$) and (iii) $k$ is built from $K$
($K\builds k$). In the context of more general ring objects, 
we say $R$ is {\em proxy-regular} if there is an
$R$-module $K$ so that (i), (ii) and (iii) hold, and we think of this
as a finiteness condition playing a similar role to that of being Noetherian.

\subsection{The Gorenstein condition}
We now say that   $S\lra k$ is {\em  Gorenstein} of shift
$a$ (and write $\shift (S)=\shift (k|S)=a$) if we have an equivalence 
$$\Hom_S(k, S)\simeq \Sigma^a k $$
of $R$-modules. More generally, we say that $S\lra R$ is 
{\em relatively Gorenstein of shift $a$} (and write $\shift (R|S)=a$) if 
$$\Hom_S(R, S)\simeq \Sigma^a R. $$

Analogously to  the classical case, we are interested in proxy-regular
rings which satisfy the Gorenstein condition.

\subsection{Gorenstein duality}
Although the Gorenstein condition itself is convenient to work with, the real reason for
considering it is the duality property that it implies. 

In classical local commutative algebra the Gorenstein duality property is
that all local cohomology is in a single cohomological degree, where it is the
injective hull $I(k)$  of the residue
field. To give a formula, we write $\Gamma_{\fm}M$ for the $\fm$-power torsion in an
$R$-module $M$, and $H^*_{\fm}(M)$ for the local cohomology of
$M$, recalling Grothendieck's theorem that if $R$ is Noetherian,
$H^*_{\fm}(M)=R^*\Gamma_{\fm}(M)$. 
The Gorenstein duality statement  for a local ring of Krull dimension $r$ therefore states
$$H^*_{\frak{m}}(R)=H^r_{\frak{m}}(R)=I(k). $$
If $R$ is a $k$-algebra, $I(k)=R^{\vee}=\Gamma_{\fm} \Hom_k(R,k)$.

Turning to ring spectra, if $R$ is a $k$-algebra we may again 
define  $R^{\vee}=\cell_k(\Hom_k(R, k))$ and observe this has the Matlis lifting property
$$\Hom_R(T, R^{\vee})\simeq \Hom_k(T,k) $$
for any $T$ built from $k$. The case when $R$ is not a $k$-algebra is
more complicated, but will not be needed here. 

In particular, if $R$ is Gorenstein of shift $a$ we have equivalences
of $R$-modules
$$\Hom_R(k, \cell_kR)\simeq \Hom_R(k,R)\stackrel{(g)}\simeq \Sigma^a
k\stackrel{(m)}\simeq \Hom_R(k, \Sigma^a R^{\vee}),  $$
where the equivalence (g) is the Gorenstein property and the
equivalence (m) is the Matlis lifting property. 
We would like to remove the $\Hom_R(k, \cdot)$ to deduce
$$\cell_kR\simeq \Sigma^a R^{\vee}. $$
Such an equivalence is known as {\em Gorenstein duality},  since
$\cell_k(R)$  is a covariant functor of $R$ and $R^{\vee}$ is a contravariant functor of
$R$. 

Morita theory \cite{DGI1} says that if $R$ is proxy-regular we may make this
deduction provided $R$ is orientably Gorenstein in the sense that the right
actions of $\cE=\Hom_R(k,k)$  on $\Sigma^ak$ implied by the two
equivalences (g) and (m) agree. This is automatic when the ring
spectrum is both a $k$-algebra and connected.

\begin{prop}
Suppose $R$ is a proxy-regular, connected $k$-algebra
and $\pi_*(R)$ is Noetherian with $\pi_0(R)=k$ and 
maximal ideal $\fm$ of positive degree elements.  If $R$ is 
Gorenstein of shift $a$, then $R$ it is automatically orientable and so 
has Gorenstein duality.  Accordingly there is a local cohomology spectral sequence
$$H^*_{\fm}(R_*)\Rightarrow \Sigma^a R_*^{\vee}. $$
\end{prop}

\begin{proof} 
First we argue that if $R$ is Gorenstein, it is automatically
orientable. Indeed, we show that $\cE$ has a
unique action on $k$. Since $R$ is a $k$-algebra, the
action of $\cE$ on $k$ factors through 
$$\cE=\Hom_R(k,k) \lra \Hom_k(k,k)=k,$$
so since $k$ is an Eilenberg-MacLane spectrum,  
the action is through $\pi_0(\cE)$. Now we observe that since $R$ is connected, 
$\Ext_{R_*}^s(k,k)$ is in degrees $\leq -s$, so that the spectral
sequence for calculating $\pi_*(\Hom_R(k,k))$ shows $\cE$ is
coconnective with $\pi_0(\cE)=k$ which must act trivially on $k$. 
\end{proof}

We note that if the coefficient ring $\pi_*(R)$ is Gorenstein and $R$ is
connective then $R$ is Gorenstein. Indeed,  the spectral sequence 
$$\Ext_{R_*}^{*,*}(k,R_*)\Rightarrow \pi_*(\Hom_R(k,R)) $$ 
collapses, to show $\pi_*(\Hom_R(k,R)) =\Sigma^a k$ for some $a$.  The $R$-module $k$ 
is characterised by its homotopy, so $\Hom_R(k,R)\simeq \Sigma^a
k$.  Conversely, if $R$ is Gorenstein, this shows that the ring
$\pi_*(R)$ has very special properties (even if it falls short of
being Gorenstein).  The following statement corrects a typographical
error in \cite[6.2]{ringlct}. 

\begin{cor} \cite{ringlct}
Suppose $R$ has Gorenstein duality of shift $a$, that $\pi_*(R)$ is
Noetherian of Krull dimension $r$ and Hilbert series $p(t)=\sum_s\dim_k(R_s)t^s$.
\begin{enumerate}
\item If $\pi_*(R) $ is Cohen-Macaulay it is also Gorenstein, and the
Hilbert series satisfies 
$$p(1/t)=(-1)^rt^{r-a}p(t).$$

\item If $\pi_*(R) $ is almost Cohen-Macaulay it is also almost Gorenstein, and the
Hilbert series satisfies 
$$p(1/t)-(-1)^rt^{r-a}p(t)=(-1)^{r-1}(1+t)q(t) \mbox{ and } q(1/t)=(-1)^{r-1}t^{a-r+1}q(t).$$ 

In any case $\pi_*(R)$ is Gorenstein in codimension 0 and almost
Gorenstein in codimension 1. 
\end{enumerate}
\end{cor}

\subsection{The relatively Gorenstein case}

We make the  elementary observation that for any ring map $\theta : S\lra R$ 
$$\Hom_R(k,\Hom_S(R,S))\simeq \Hom_S(k,S).$$
Thus we conclude that if $S\lra R$ is relatively Gorenstein then $R$ is
Gorenstein  if and only if  $S$ is Gorenstein, and in that case 
$$\shift (k|S)=\shift (k|R)+\shift (R|S).$$

\begin{example}
\label{eg:relGorko}
 The ring map $S=ko\lra ku=R$ is relatively Gorenstein of shift
2. Indeed, the connective version of Wood's theorem states that there
is an equivalence $ku\simeq ko\sm
(S^0\cup_{\eta} e^2)$ of $ko$-modules, so that  
$$\Hom_{ko}(ku, ko)\simeq \Sigma^{-2} ku. $$
Since $ku_*=\Z [v]$ we see that $ku$  is Gorenstein of shift $-4$ over
$\Ftwo$,  and it follows that 
$ko$ is Gorenstein of shift $-6$ over $\Ftwo$. 
\end{example}

\begin{example}
\label{eg:relGortmf}
 Precisely similar statements hold for $tmf$. This is based on results of
Hopkins-Mahowald \cite{HM}, with an improved formal context of
Hill-Lawson \cite{HillLawson} giving maps in the category of  commutative $tmf$
algebras. The results about finite cell complexes are proved by Mathew
\cite{AM}. 

As background we note that at primes $p\geq 5$, we have
$tmf_*=\Z_{(p)}[c_4,c_6]$ with $c_4$ of degree 8 and $c_6$ of degree
12. It is therefore immediate from the coefficients that $tmf$ is
Gorenstein of shift $-23$ over $\Fp$. The primes 3 and 2 are more
interesting.

(i) At the prime 3, we consider the  map $tmf \lra tmf_1(2)$ of
commutative $tmf$-algebras \cite[Theorem 6.1]{HillLawson}. There is an
equivalence of $tmf$-modules  
$$tmf_1(2)\simeq tmf \sm (S^0\cup_{\alpha_1} e^{4}\cup_{\alpha_1}
e^{8})$$
 (\cite[Theorem 7.7]{AM} gives an equivalence of spectra. Writing 
$T=S^0\cup_{\alpha_1} e^{4}\cup_{\alpha_1}e^{8}$, a map $f:T\lra
tmf_1(3)$ determines a map $tmf \sm T \lra tmf_1(3)$ of
$tmf$-modules. To see that the map is an equivalence we may 
check it is an isomorphism in mod $3$ cohomology, and for this 
we only need to check it is an epimorphism in mod 3 cohomology.
Since $tmf_1(3)\simeq BP\langle 2\rangle \vee \Sigma^8 BP\langle
2\rangle$ only two generators are required over the Steenrod algebra,
and  it suffices to choose $f $ so that generators in degrees 0 and 8 
are in the image).   It follows that  
$$\Hom_{tmf}(tmf_1(2),tmf )\simeq \Sigma^{-8} tmf_1(2). $$
Since $tmf_1(2)_*=\Z_{(3)}[c_2, c_4] $ (where $|c_i|=2i$) we see that $tmf_1(2)$ is Gorenstein of shift
$-15$.  Hence we deduce by Gorenstein descent  that $tmf\lra \mathbb{F}_3 $ is
Gorenstein of shift $-23$. 

(ii) At the prime 2, we consider the map $tmf \lra tmf_1(3)$
\cite[Theorem 6.1]{HillLawson} of commutative $tmf$-algebras. Here $tmf_1(3)$ is
a form of $BP\langle 2\rangle$ (previously proved to have a
commutative model by Lawson-Naumann \cite{LawsonNaumann1, LawsonNaumann2}) and there is
an equivalence of $tmf$-module spectra 
$$tmf_1(3)\simeq tmf \sm DA(1)$$
(again, the equivalence of spectra is given in \cite[Theorem
6.6]{AM}.  A map $DA(1)\lra tmf_1(3)$ determines a map of
$tmf$-modules, and as above it suffices to show the resulting $tmf$-module map 
$tmf\sm DA(1) \lra tmf_1(3)$ is an epimorphism in mod 2
cohomology. Since the mod 2 cohomology is generated in degree 0 over
the Steenrod algebra, this is easily arranged).  It follows that  
$$\Hom_{tmf}(tmf_1(3),tmf )\simeq \Sigma^{-12} tmf_1(3). $$
Since $tmf_1(3)_*= \Z_{(2)}[\alpha_1, \alpha_3]$ (where $|\alpha_i|=2i$) we see that $tmf_1(3)$ is Gorenstein of shift
$-11$.  Hence we deduce by Gorenstein descent  that $tmf \lra \Ftwo$ is
Gorenstein of shift $-23$. 
\end{example}

In general, it can be difficult to decide if $S\lra R$ is relatively Gorenstein,
and we prefer to give conditions depending on $Q$. 

\subsection{Gorenstein Ascent}
In effect the Gorenstein Ascent theorem will state that under suitable
hypotheses  (see Section \ref{sec:Gorascent}) there is an equivalence
$$\Hom_R(k,R)\simeq \Hom_{Q}(k, \Hom_S(k,S)\tensor_k Q). $$
When this holds, it follows that if $S$ and $Q$ are Gorenstein, so is
$R$ and 
$$\shift (R)=\shift (S)+\shift (Q).$$

\subsection{Arithmetic of shifts}
We summarize the behaviour of Gorenstein shifts  in the ideal
situation when ascent and descent both hold. If all rings and maps are Gorenstein of the indicated
shifts
$$\stackrel{a} S \stackrel{\lambda}\lra 
\stackrel{b} R \stackrel{\mu}\lra 
\stackrel{c} Q$$
then $b=a+c, \lambda =-c$ and $\mu=a$

\section{Some known calculations}
\label{sec:egs}
The paper is motivated by several calculations when $S=\bbS$ is the sphere
spectrum. 

\begin{example} {\bf (The map $R=\Fp \lra \Fp=k$.)}
\label{eg1:Fp}
We consider the homotopy of  $THH(\Fp)$.  
B\"okstedt \cite{Bok} has calculated  $THH_*(\Fp )=\Fp  [\mu_{2}]$. 
The ring $THH_*(\Fp)$ is Gorenstein of shift  
$-3$. 
\end{example}

\begin{example} {\bf (The map $R=\Z \lra \Fp=k$.)}
\label{eg1:Z}
We consider the  mod $p$ homotopy of $THH(\Z)$. B\"okstedt \cite{Bok}
has calculated $THH_*(\Z;\Fp)=\Fp [\mu_{2p}]\tensor \Lambda_{\Fp} (\lambda_{2p-1})$. 
The ring  $THH_*(\Z; \Fp)$ is  Gorenstein of shift
$(-2p-1)+(2p-1)=-2$ (B\"okstedt Duality). The ring but not the shift depends on $p$.


The calculations of Lindenstrauss and Madsen \cite[4.4]{LM} have a
similar pattern. Indeed, if $\cO$ is a ring of integers in a number field, which is either
unramified or wildly ramified, $THH_*(\cO , \cO /p)$ is polynomial tensor
exterior (over $\cO /p$) on generators differing in degree by 1. In the
tamely ramified case  the ring $THH_*(\cO, \cO /p)$ is more complicated, but it is
still Gorenstein of shift $-2$. 
\end{example}

\begin{example} {\bf (The map $R=lu \lra \Fp=k$.)}
\label{eg1:lu}
We consider mod $v_1,p$ homotopy of $THH (lu)$ where $lu$ is the Adams
summand of $p$-local connective $K$-theory with coefficients
$lu_*=\Z_{(p)}[v_1]$.  McClure-Staffeldt \cite{MS} (see also
Ausoni-Rognes \cite{AR}) have calculated
  $THH_*(lu;\Fp)=\Fp [\mu_{2p^2}]\tensor \Lambda_{\Fp}
 (\lambda_{2p-1}, \lambda_{2p^2-1})$. 
The ring $THH_*(lu; \Fp)$ is Gorenstein of shift 
$(-2p^2-1)+(2p-1+2p^2-1)=2p-3$. 
\end{example}

\begin{example} {\bf (The map $R=ku \lra ku/(p,v_1)=k$.)}
\label{eg1:kukupvone}
For primes $p>2$, Ausoni calculates the mod $p,v_1$ homotopy of $THH(ku)$ 
and shows that 
$THH_*(ku;ku/(p,v_1))= \Lambda (\lambda_{2p-1})\tensor
\Fp[\mu_{2p^2}]\tensor Q, $ where $Q$ is Poincar\'e duality algebra of 
formal dimension $2p^2-1$ \cite[9.15]{Ausoni}. Although $\pi_*(ku/(p,v_1))=\Fp
[v]/(v^{p-1})$ is not a field,  we may make still consider duality
properties over $\Fp$. 
The ring  $THH_*(ku; ku/(p,v_1))$ has Gorenstein duality  
over $\Fp$ with shift  $(-2p^2-1)+(2p-1)+(2p^2-1)=2p-3$ (Ausoni Duality). This striking
example stimulated the author to investigate Gorenstein duality for $THH$. 
\end{example}

\begin{example} {\bf (The map $R=ko \lra H\Ftwo=k$.)}
\label{eg1:ko}
Angeltveit and Rognes \cite{AR} show that 
$THH_*(ko;\Ftwo )= \Lambda (\lambda_{5}, \lambda_7 )\tensor
\Ftwo [\mu_8]. $ The ring $THH_*(ko;\Ftwo)$ is Gorenstein of shift $5+7-8-1=3$. 
\end{example}

\begin{example} {\bf (The map $R=tmf \lra H\Ftwo=k$.)}
\label{eg1:tmf}
It is easily deduced from the calculations of Angeltveit and Rognes \cite{AR} that 
$THH_*(tmf;\Ftwo )= \Lambda (\lambda_9, \lambda_{13}, \lambda_{15} )\tensor
\Ftwo [\mu_{16}]. $ The ring $THH_*(tmf;\Ftwo)$  is Gorenstein of shift $9+13+15-16-1=20$. 
\end{example}

\section{Gorenstein ascent}
\label{sec:Gorascent}

We have begun to see the value of understanding the behaviour of the
Gorenstein condition in cofibre sequences, and we turn to a more
systematic discussion.

 We suppose that $S\lra R \lra Q$ is a cofibre sequence of commutative
algebras with a map to $k$, and we now consider the Gorenstein ascent
question.  When does the fact that $S$ is Gorenstein imply that $R$ is
Gorenstein? It is natural to assume that $Q$ is Gorenstein, but it is
known this is not generally sufficient. We identify a number of
circumstances in which it is sufficient, and  in characteristic $p$ we
give a useful general result. Before we do this, we look at the finiteness conditions. 

\subsection{Proxy-regularity}
We provide a tool for proxy-regular ascent. It seems that some
hypothesis is necessary and we give one in a form applying to cases of
interest here. 

First, we should introduce notation for the standard Koszul complex
associated to a sequence $r_1, r_2, \ldots , r_n$ of elements of
elements of  $\pi_*R$. For $x \in \pi_*(R)$ we define $K(R;x)$ by 
the cofibre sequence
$$R\stackrel{x} \lra R \lra K(R;x), $$
and now we take
$$K(R;r_1, \ldots , r_n)=K(R;r_1)\tensor_R \cdots \tensor_R K(R;r_n). $$
In the usual way, a concrete realization requires the choice of
specific cocycle representatives, but the homotopy type does not
depend on these choices. If $R$ is a classical ring, this gives the
standard construction  $K(R;x)=[R\stackrel{x}\lra R]$ with the copies
of $R$ in degrees 0 and 1.

\begin{lemma}
\label{lem:proxyascent}
Suppose that $S$ is proxy-regular with Koszul complex $K_S$ and that $Q$ has a Koszul complex of
the special form $K_Q=K(Q; q_1, \ldots , q_n)$ where $q_i\in \pi_*(Q)$ lifts to $r_i
\in \pi_*(R)$ for $i=1, \ldots , n$. Then $R$ is proxy-regular with Koszul complex
$K_R:=K_S\tensor_S K(R;r_1, \ldots, r_n)$. 
\end{lemma}
\begin{proof} There are three things to prove.

 Since $S\finbuilds K_S$ and $R\finbuilds K(R;r_1, \ldots, r_n)$, it
follows that 
$$R=S\tensor_SR\finbuilds K_S\tensor_S K(R;r_1, \ldots, r_n)=K_R.$$ 

 Since $Q\simeq k\tensor_SR$, we find firstly 
$$k\finbuilds K_Q= k\tensor_S K(R;r_1, \ldots , r_n)\finbuilds
K_S\tensor_S K(R;r_1, \ldots , r_n)=K_R$$
and secondly 
$$K_R=K_S\tensor_SK(R; r_1, \ldots , r_n) \builds k\tensor_SK(R; r_1, \ldots
, r_n) =K_Q\builds k .$$
This completes the proof. 
\end{proof}

\subsection{Good approximation implies ascent}

The core of our results about ascent come  from \cite{DGI1}. Indeed, 
the proof of \cite[8.6]{DGI1} gives a sufficient condition for Gorenstein
ascent in the commutative context. 

\begin{lemma}
\label{lem:Gorascent}
If $S$ and $R$ are commutative and the natural map $\nu:
\Hom_S(k,S)\tensor_S R\lra  \Hom_S(k,R)$ 
is an equivalence then 
$$\Hom_R(k,R)\simeq \Hom_{Q}(k, Hom_S(k,S)\tensor_k Q). $$
In this case, if  $S$ and $Q$
are Gorenstein, so is $R$, and the shifts add up: $\shift (R)=\shift
(S)+\shift (Q)$. \qqed
\end{lemma}

Now that we have a sufficient condition for Gorenstein ascent, we want
to identify cases in which  $\nu$ is an equivalence. 
The most familiar case is when $R$ is small over $S$ (or
equivalently, when $Q$ is finitely built from $k$). 
We emphasize that the hypothesis on $\nu$ in Lemma \ref{lem:Gorascent}
only depends on $R$ as a {\em module} over $S$, and we will obtain a
useful generalization by approximating $R$ by $S$-modules for which 
$\nu$ is an equivalence. The approximation will be as an inverse
limit, and to see the approximation is accurate we need to impose
hypotheses to ensure inverse limits and tensor products commute. 

\begin{lemma}
\label{lem:limtensor}
Suppose $M$ and $N$ are $S$-modules and $N\simeq \holim_n N_n$, and
consider the natural map 
$$\kappa: M\tensor_S [\holim_n N_n]\lra \holim_n[M\tensor_S N_n].$$

The map $\kappa $ is an equivalence in either of the following
circumstances
\begin{itemize}
\item $S$ is connective, $M$ is of upward finite type and the modules
$N_n$ are uniformly bounded below. The hypothesis on $M$ holds if
$\pi_0(S)=k$, 
$\pi_*(M)$ is bounded below and $\pi_n(M)$ is finite dimensional over
$k$ for all $n$.
\item $S$ is coconnective, $M$ is of downward finite type and the modules
$N_n$ are uniformly bounded above. The hypothesis on $M$ holds if
$S$ is simply coconnected,  $\pi_*(M)$ is bounded above and $\pi_n(M)$ is finite dimensional over
$k$ for all $n$.
\end{itemize}
\end{lemma}

\begin{proof}
In the first part the fact that $M$ is of upward finite type and the $N_n$
are uniformly bounded below is enough to see that the limit is achieved
in each degree.  It is proved as \cite[3.13]{DGI1} that the homotopy
level condition ensures $M$ is of upward finite type. 

The proof of the second part is precisely similar, with a reference to
\cite[3.14]{DGI1}. 
\end{proof}

\begin{lemma}
\label{lem:Gorascentlim}
Suppose that $\pi_*(S)$ is  Noetherian and that 
 $\pi_*(\Hom_S(k,S))$ is a finitely generated module over $\pi_*(S)$
 and that $R\simeq \ilim_n R_n$ for small $S$-modules $R_n$.
The hypothesis  of Lemma \ref{lem:Gorascent} applies in either of the following
circumstances
\begin{itemize}
\item $S$ is connected and the $R_n$ are uniformly bounded below
\item $S$ is simply coconnected and the $R_n$ are uniformly bounded above
\end{itemize}
In this case, 
$$\Hom_R(k,R)\simeq \Hom_{Q}(k, Hom_S(k,S)\tensor_k Q) $$
and Gorenstein ascent holds for the cofibre sequence $S\lra R\lra Q$.    
\end{lemma} 

\begin{proof}
First, 
$$\Hom_S(k, R) \simeq \Hom_S(k, \holim_n R_n) \simeq \holim_n
\Hom_S(k,  R_n).  $$ 
Now, since $R_n$ is a small $S$-module,  $\Hom_S(k,R_n)\simeq \Hom_S(k,S)\tensor_S R_n$.

It therefore remains to show that the natural map 
$$\kappa: M\tensor_S [\holim_n R_n] \lra \holim_n[M\tensor_S R_n] $$
is an equivalence when $M=\Hom_S(k,S)$ so the conclusion follows from 
Lemma \ref{lem:limtensor}. 
\end{proof}

\subsection{Building good approximations}
We give criteria under which  $R$ may be approximated in this
way. First, we assume that $R$ is a $k$-algebra, and  it is convenient to
introduce some further terminology.

\begin{defn}
We say that a map $R \lra Q$ is {\em \pifinite\ } if $\pi_*(Q)$ is
finitely generated as a module over  $k [x_1, \ldots , x_n]$ for some
finite set of elements $x_1, \ldots, x_n$ of $ \pi_*(R)$. A cofibration sequence $S\lra R
\lra Q$ is {\em \pifinite } if the map $R\lra Q$ is \pifinite .
\end{defn}

\begin{remark} 
 If
$\pi_*(R)$ is Noetherian, it is equivalent to ask that $\pi_*(Q)$ is
finitely generated over $\pi_*(R)$. 
\end{remark}

\begin{prop}
\label{prop:GorascentNoeth}
Suppose that $\pi_*(S)$ is  Noetherian,  $\pi_*(R)$ and 
$\pi_*(\Hom_S(k,S))$ are  finitely generated $\pi_*(S)$-modules, $R$
is a $k$-algebra and the cofibration is \pifinite , and suppose that either (i) $S, R$ and
$Q$ are all  connected or (ii) that $S, R$ and $Q$ are all  coconnected and $S$ is
simply  coconnected.  Under these conditions, 
$$\Hom_R(k,R)\simeq \Hom_{Q}(k, Hom_S(k,S)\tensor_k Q) $$
and Gorenstein ascent holds for the cofibre sequence $S\lra R\lra Q$.    
\end{prop}

\begin{proof} From the \pifinite\ hypothesis, by the Noether
  normalization argument,   there is a polynomial subring $R(1)_*$ of $\pi_*(R)$ over which $\pi_*(Q)$ is finitely
  generated.
 Now  let $R(2)_*, R(3)_*, \ldots$ be the subrings
generated by the $2$nd, $4$th, $8$th .... powers of generators of $R(1)_*$.

 Next we construct  (non-commutative) $k$-algebra spectra $R(n)$ with 
$\pi_*R(n)=R(n)_*$.    For a polynomial ring on a single generator of
degree $d$, we
can consider the James construction $J_k(S^d)$ on a sphere over
$k$. This is the free associative $k$-algebra spectrum on the $d$-sphere
and has homotopy $k[X_d]$.  If $R(n)_*=k[x_1, \cdots, x_s]$, we form 
$$J_k(S^{d_1})\tensor_k \cdots \tensor_k J_k(S^{d_s})$$
where $x_i$ is of degree $d_i$.  If $A$ is a commutative $k$-algebra  then we
may construct a ring map taking $X_i$ to $x_i$ as the composite
$$J_k(S^{d_1})\tensor_k \cdots \tensor_k J_k(S^{d_s})
\lra A\tensor_k\cdots\tensor_k A\lra A. $$
The first map takes $X_i$ to $x_i$ in the $i$th factor, and 
 is a map of associative rings. The second map is multiplication in
 $A$, and this is a ring map since $A$ is commutative. 

Using these ring spectra $R(n)$,  using tensor products of the single
variable case as above, we may construct maps
$$
\diagram
\ldots  \rto &R(2) \rto &R(1) \rto&R\rto &Q
\enddiagram$$
realizing the algebras we took  in homotopy. 
Now by construction $Q_n=Q\tensor_{R(n)}k$  is finitely built from $k$
since its homotopy is a finite dimensional $k$-vector space. Since the 
polynomial generators were in increasingly large degrees, 
 $Q=\holim_n Q_n$.  Similarly, if we write $R_n=R\tensor_{R(n)}k$ 
we have $R\simeq \holim_n R_n$, and  $k\tensor_SR_n\simeq 
k\tensor_S R\tensor_{R(n)} k\simeq
Q\tensor_{R(n)}k=Q_n$. Thus we have 
sequences $S\lra R_n \lra Q_n$ and $R_n$ is small as an $S$-module. 

The result now follows from Lemma \ref{lem:Gorascentlim}. 
\end{proof}

To apply this, we first note that  in characteristic $p$
 the \pifinite\  condition is automatic when $\pi_*(S)$ is in a finite
range of degrees. 

\begin{lemma}
\label{lem:charppifinite}
Suppose  $S\lra R\lra Q$ is a cofibre sequence, either connected and
satisfying the hypotheses of Lemma \ref{lem:connSS} or coconnected and
satisfying the hypotheses of Lemma \ref{lem:coconnSS}. Suppose in
addition that the cofibre sequence is one of $k$-algebras, where $k$ is
a field of characteristic $p>0$  and $\pi_*(S)$ is Noetherian
and in a finite range of degrees, then the cofibre sequence is
\pifinite .
\end{lemma}

\begin{proof}
Using the spectral sequence of Lemma \ref{lem:connSS} or Lemma
\ref{lem:coconnSS} as appropriate we see that if $x\in \pi_*(Q)$ survives to the $r$th
page then $d_r(x^p)=0$, so that $x^p$ survives to the $(r+1)$st page. 
If $\pi_*(S)$ is in a finite range of degrees, the spectral sequence
collapses at the $N$th stage for some $N$ and the $p^{N-1}$th powers
of all elements survive, and therefore lie in the image of
$\pi_*(R)\lra \pi_*(Q)$ so that the cofibration is \pifinite . 
\end{proof} 

We may now apply  Proposition \ref{prop:GorascentNoeth} to give a
useful characteristic $p$ Gorenstein ascent theorem. 

\begin{cor}
\label{cor:GorascentNoeth}
Consider a  cofibre sequence $S\lra R\lra Q$ of $k$-algebras as in
Lemma \ref{lem:charppifinite}. 
Suppose that $\pi_*(S)$  Noetherian and either connected or simply
coconnected and $k$ is a field of of characteristic $p>0$. If $\pi_*(R)$ and 
$\pi_*(\Hom_S(k,S))$ are  finitely generated $\pi_*(S)$-modules, and 
$\pi_*(S)$ is concentrated in a finite range of degrees then 
$$\Hom_R(k,R)\simeq \Hom_{Q}(k, Hom_S(k,S)\tensor_k Q) $$
and Gorenstein ascent holds for the cofibre sequence.  
\end{cor}

\section{Ausoni-B\"okstedt duality. }
\label{sec:ABduality}
We now have the necessary ingredients to state and prove our duality
result. The idea is that if we are given  maps $\Ca\lra \Ba \lra k$ of
commutative ring spectra with the cofibre ring spectrum $A=\Ba \tensor_\Ca k$
  Gorenstein then (at least under some hypotheses on $k$ and $A$)
 if  $\TbB= THH(\Ba;k)$ is  Gorenstein then $\TbC=THH(\Ca;k)$ is also
 Gorenstein. Since we are deducing the domain $\TbC$  is Gorenstein from
  the fact that  $\TbB$  is Gorenstein,  we think of this as a {\em descent} theorem, even
  though the principal ingredient is an {\em ascent} theorem for a suitable
  cofibration. 

\subsection{Gorenstein descent for $THH$}

The key to method is the existence of a suitable cofibration sequence
conjectured on the basis of the examples and proved by Dundas.

\begin{lemma} (Dundas)
\label{lem:cofibresequence}
 Given a cofibre sequence 
$\Ca\lra \Ba \lra A$  of commutative ring spectra over $k$ (i.e., $\Ba$
has a map to $k$ and $A=\Ba\tensor_{\Ca} k$)
there is a cofibre sequence of commutative $k$-algebra spectra 
  $$A \lra \TbC\lra \TbB$$
where $\TbC=\Ca\tensor_{\Ca^e} k=THH(\Ca;k)$ and $\TbB=\Ba \tensor_{\Ba^e} k
=THH(\Ba; k)$.
\end{lemma}
\begin{remark}
Dundas's lemma makes $THH( \cdot ; k)$ remarkably computable. For
example, Lindenstrauss points out that if $R$ is $k$-algebra, we may apply the 
Dundas Lemma to the cofibre sequence $k\lra R \lra R$ to deduce
$$THH_*(R;k)\cong THH_*(k)\tensor \mathrm{Tor}_*^R(k,k) .$$
In particular, this allows one to deduce from B\"okstedt's calculation
that  $THH_*(k)=k[\mu_2]$ for any field $k$ of characteristic $p$. 
\end{remark}

\begin{remark}
\label{rem:overS}
Dundas's lemma applies also to  $HH_{\bullet}( \cdot | S ; k)$ when
$S$ is less complicated than the sphere spectrum. For example we may take $S=C^*(Z)$
for a space $Z$, and suppse given a map $ Y \lla X$ of simply
connected spaces over $Z$ with fibre $F$, giving 
$$\left( C\lra B \lra A\right) =\left( C^*(Y)\lra C^*(X)\lra C^*(F) \right).$$
 We see 
$$HH_{\bullet}(C^*(Y)|C^*(Z); k)=C^*(fibre (Y\lra Y\times_ZY))$$ 
provided $Y\times_ZY$ is simply connected,  so that if $X$ also satisfies
the corresponding hypothesis, Dundas's lemma gives a cofibre sequence
$$C^*(F) \lra C^*(fibre (Y\lra Y\times_ZY)) \lra C^*(fibre (X\lra
X\times_ZX)). $$
If $Z=*$ this comes from the fibre sequence 
$$F\lla \Omega Y \lla   \Omega X $$
obtained from the Puppe sequence generated by $Y\lla X$.
\end{remark}

\begin{proof}
For a $\Ca$ bimodule $M$, according to the original
definition,  the topological Hochschild
homology  $THH(\Ca;M)$ is a realization of the Hochschild simplicial
spectrum  with $n$th term $M \tensor_{\bS} \Ca^{\tensor  n}$,
where the tensor power is for $\tensor_{\bS}$;
this is natural for maps of rings and of bimodules. In particular, the
map $\Ca\lra \Ba$ of ring spectra vertically and the map 
$\Ba \tensor_{\bS} k \lra k$ of $\Ba$-bimodules horizontally give a commutative square
$$\begin{array}{ccc}
\Ba \tensor_{\bS} k \tensor_{\bS}
\Ca^{\tensor n}&\rightarrow&k\tensor_{\bS} \Ca^{\tensor n}\\
\downarrow && \downarrow\\
\Ba \tensor_{\bS} k \tensor_{\bS}
\Ba^{\tensor n}&\rightarrow&k\tensor_{\bS} \Ba^{\tensor n}
\end{array}$$
which is evidently a pushout square of commutative ring spectra. 
 Taking geometric realizations we obtain  the pushout square
$$\begin{array}{ccc}
THH (\Ca; \Ba\tensor_{\bS}k )&\rightarrow&THH(\Ca; k)\\
\downarrow && \downarrow\\
THH (\Ba; \Ba\tensor_{\bS}k )&\rightarrow&THH(\Ba; k).
\end{array}$$
Now we use the fact that $THH$ with coefficients in a  bimodule of the
form $M=X\tensor_{\bS}Y$  simplifies:
$$THH(\Ca; X\tensor_{\bS}Y)=Y\tensor_\Ca X. $$
The pushout square now gives the required result. 
\end{proof}

We now want to take the cofibre sequence $A\lra \TbC \lra \TbB$ and deduce
that when $A$ and $\TbB$ are Gorenstein, so is $\TbC$. We need only verify 
that the hypotheses of Lemma \ref{lem:proxyascent} and
Corollary \ref{cor:GorascentNoeth} are satisfied.

\begin{thm} {\em (Gorenstein descent for THH)}
\label{thm:THHGordescent}
Suppose $\Ca\lra \Ba\lra A$ is a cofibre sequence of connective commutative ring spectra
with maps to $k$ and that 
\begin{enumerate} 
\item $A$ and $\TbB$ are proxy-regular and  Gorenstein 
\item $\pi_*(A)$ and $\pi_*(\TbB)$ are Noetherian
\item the map $\TbC \lra \TbB$ is \pifinite\
\end{enumerate} 
then $\TbC$ is proxy-regular and  Gorenstein with 
$$\shift (\TbC)=\shift (\TbB)+\shift (A).$$
\end{thm}

\begin{proof}
We consider the cofibration $A \lra \TbC \lra \TbB$ of Lemma
\ref{lem:cofibresequence}. 


By the \pifinite\ hypothesis we may choose a finite number of elements of  $\pi_*(\TbC)$ so
that $\pi_*(\TbB)$ is finitely generated over the $k$-algebra they
generate, and we may form a Koszul complex by using these
generators. This verifies the hypotheses
necessary to see that $\TbC$ is proxy-regular by Lemma
\ref{lem:proxyascent}. 

The hypotheses for Gorenstein ascent from $A$ to $\TbC$ are stated explicitly. 
\end{proof}

\begin{cor}
\label{cor:gor}
If $\Ca$ is Gorenstein of shift $a$ and augmented over a field $k$ of
characteristic $p$ and if
$\Ca$ is regular  then $THH(\Ca;k) \lra k$ is 
Gorenstein  of shift $-a-3$. 
\end{cor}

\begin{proof}
We apply the theorem to the cofibre sequence $\Ca\lra k \lra A$. 
Since $\Ca$ is regular, $\pi_*A$ is finite dimensional and hence Noetherian. 
Now note that $k$ is Gorenstein of shift 0, so by Gorenstein Ascent
for $\Ca\lra k \lra A$, the ring $A$ is Gorenstein and we have $\shift
(A)=-a$.

From B\"okstedt's calculation $\TbB = THH(k)$ has Noetherian homotopy $k[\mu_2]$, so it is
Gorenstein of shift $-3$.  Now observe that by Lemma
\ref{lem:charppifinite},  the cofibration
$A\lra \TbC\lra \TbB$ is \pifinite . 
Thus the hypotheses of Theorem \ref{thm:THHGordescent} are satisfied, and we may  apply
 Gorenstein Ascent  to the cofibration $A\lra \TbC \lra \TbB$ to obtain the conclusion. 

\end{proof}

\begin{remark}
\label{rem:overk}
As in Remark \ref{rem:overS}, the same argument applies if $S$ less
complicated than the sphere spectrum. In particular if we have a map $C\lra B$ of
augmented $k$-algebras, Dundas's Lemma supplies a cofibre sequence
$$A \lra HH_{\bullet}(C|k;k)\lra HH_{\bullet}(B|k;k). $$
If $B=k$ we find 
$$k\tensor_Ck\simeq HH_{\bullet}(C|k;k). $$
For instance if $C=C^*(X)$ this corresponds to the fact that $\Omega
X$ is the fibre of the diagonal $X\lra X\times X$. 
\end{remark}

\section{Examples}
\label{sec:egs2}
We observe that Theorem \ref{thm:THHGordescent} gives  a
non-calculational proof of several of the dualities we
observed in Section \ref{sec:egs}  above, as well as giving many new
examples where the coefficient rings are not known.  In each case we specify $\Ca\lra \Ba$ and $k$ and then
discuss the resulting cofibre sequence $A\lra THH(\Ca;k)\lra
THH(\Ba;k)$. 

\subsection{Known examples revisited}
We do not add to the explicit calculations described in Section
\ref{sec:egs} above, but we emphasize that the only calculational
input is B\"okstedt's theorem. Interesting structural relationships
are highlighted by this approach.

\begin{example} {\em  (Example \ref{eg1:Fp} revisited: $\Fp$.)}
\label{eg2:Fp}
If we take $\Ca \lra \Ba$ to be $\Fp \lra \Fp$ and $k=\Fp$, we find $A=\Fp$. It is
immediate that $\Fp$ is small over $\Fp $ and 
$\Fp \lra \Fp$  is Gorenstein of shift 0. Corollary \ref{cor:gor}
shows $THH (\Fp)$ is Gorenstein of shift  $-0-3=-3$. 
\end{example}

\begin{example} {\em (Example \ref{eg1:Z} revisited: $\Z$)}
\label{eg2:Z} 
If we take $\Ca\lra \Ba$ to be $\Z\lra \Fp$ and $k=\Fp$,  we find $A
 \sim C_*(S^1)$ (where $\sim$ means that the coefficient rings are
 isomorphic).
It is immediate that $\Fp$ is small over $\Z$ and easy to check that $\Z\lra \Fp$ 
is Gorenstein of shift $-1$.  Corollary \ref{cor:gor} shows that  $THH(\Z;
\Fp)$ is Gorenstein of shift   $1+(-3)=-2$. 

The spectral sequence of Lemma \ref{lem:connSS} gives an alternative approach to the
calculational proof. The only necessary input would be to know that the differential
$d_2(\mu_2)\neq 0$. This then shows that $d_2(\mu_2^n)\neq 0$ unless
$n$ is a multiple of $p$, so that the $E_3=E_{\infty}$ term is
generated by $\mu_2^p$ (giving $\mu_{2p}$) and $\mu_2^{p-1}\tau$ 
(giving $\lambda_{2p-1}$). 

The Lindenstrauss-Madsen example in the unramified case can be treated
in the same way, since $\cO /p=k$. However, the ramified case is not
covered by our analysis since $THH_*(\cO /p)$ is not Noetherian. 
\end{example}

\begin{example}
{\em (Example \ref{eg1:lu} revisited: $lu$)}
\label{eg2:lu}
If we take $\Ca\lra \Ba$ to be $lu\lra \Z$ and $k=\Fp$,  we find $A\sim
C_*(S^{2p-1})$. It is easy to check that $\Fp$ is small over $lu$ and that $lu\lra \Fp$ 
is Gorenstein of shift $-(2p-2)-1-1=-2p$.  From B\"okstedt duality for $THH(\Z ; \Fp)$, Theorem
\ref{thm:THHGordescent} shows  $THH(lu;\Fp)$ is Gorenstein of
shift $(2p-1)+(-2)=2p-3$. 

The spectral sequence of Lemma \ref{lem:connSS} gives an alternative approach to the
proof. The only necessary input would be  to know that the differential
$d_{2p}(\mu_{2p})\neq 0$. This then shows that $d_{2p}(\mu_{2p}^n)\neq 0$ unless
$n$ is a multiple of $p$, so that the $E_{2p+1}=E_{\infty}$ term is
generated by $\mu_{2p}^p$ (giving $\mu_{2p^2}$), $\mu_{2p}^{p-1}\tau_{2p-1}$ 
(giving $\lambda_{2p^2-1}$) and $\lambda_{2p-1}$ (which survives as it
is).

Of course we get the same conclusion by taking $\Ca\lra \Ba$ to be
$lu \lra \Fp$ and $k=\Fp$. In that case we find $A\sim C_*(S^1\times S^{2p-1})$ so
that by Corollary \ref{cor:gor} $THH(lu;\Fp)$ is Gorenstein of shift $[(2p-1)+1]-3=2p-3$. 
\end{example}

\begin{example}
{\em (Example \ref{eg1:kukupvone} revisited: $ku$)}
\label{eg2:kukupvone}
If we take $\Ca\lra \Ba$ to be $ku\lra ku/(p,v_1)$ and $k=ku/(p,v_1)$,
we find $A \sim C_*(S^1\times S^{2p-1}; ku/(p,v_1))$ (a Poincar\'e duality
algebra of formal dimension $4p-4$). It is easy to see that $ku/(p,v_1)$ is small over $ku$ and $ku\lra ku/(p,v_1)$ 
is Gorenstein of shift $-2p$. 

 In order to proceed we would need to know that $ku/(p, v_1)$ has a commutative ring model and
that $THH(ku/(p,v_1))$ is Gorenstein of shift $x$. 
 We would then deduce  $THH(ku;ku/(p,v_1))$ is Gorenstein of shift $4p-4+x$. 
\end{example}

\begin{example}
{\em (Example \ref{eg1:ko} revisited: $ko$)}
\label{eg2:ko}
We take $\Ca\lra \Ba$ to be $ko\lra ku$ and $k=\Ftwo$. As in the
discussion of the relatively Gorenstein condition (Example
\ref{eg:relGorko}) Wood's Theorem gives the cofibre sequence $\Sigma ko\stackrel{\eta}
\lra ko \lra ku$, showing  that $ku$ is small over $ko$ and $A\sim C_*(S^2)$. 

Since $THH(ku; \Ftwo) $ is Gorenstein of shift 1 by Example
\ref{eg2:ku},  and the complex in Wood's Theorem is self dual of
dimension 2, Theorem \ref{thm:THHGordescent}
shows $THH(ko; \Ftwo) $ is Gorensten of shift $1+2=3$. 

More directly, using Example \ref{eg:relGorko} we could take $\Ca \lra \Ftwo$ to be $ko\lra \Ftwo$,
and apply Corollary \ref{cor:gor} to conclude that $THH(ko;\Ftwo)$ is 
Gorenstein of shift $-(-6)-3=3$
\end{example}

\begin{example}
{\em (Example \ref{eg1:tmf} revisited at $p=3$: $tmf$ localized at 3)}
\label{eg2:tmftwo}
We take $\Ca\lra \Ba$ to be $tmf\lra tmf_1(2)$ at the prime 3 and
$k={\mathbb{F}}_3$.  As in the discussion of the relatively Gorenstein
condition (Example \ref{eg:relGortmf}(i)) the $tmf$-module
$tmf_1(2)$ is $tmf$ extended by a three 
cell  complex, so it is small. We have also seen $tmf_1(2)$ is Gorenstein of shift
$-15$, so that by Theorem \ref{thm:THHGordescent} the ring $THH(tmf_1(2); {\mathbb{F}}_3)$ is Gorenstein of shift
$12$. Since the three cell complex is self-dual of dimension 8,
Theorem \ref{thm:THHGordescent} shows that  $THH(tmf; {\mathbb{F}}_3))$ is Gorenstein of
shift $20$. 

More directly, using Example \ref{eg:relGortmf} (i) we observe that since $tmf_1(2)$ is small over
$tmf$ and ${\mathbb{F}}_3$ is small over $tmf_1(2)$ then 
${\mathbb{F}}_3$ is small over $tmf$. Since $tmf$ is Gorenstein
of shift $-23$ we may apply Corollary
\ref{cor:gor} to deduce $THH(tmf; {\mathbb{F}}_3))$ is Gorenstein of
shift $20$. 
\end{example}

\begin{example}
{\em (Example \ref{eg1:tmf} revisited at $p=2$: $tmf$ localized at 2)}
\label{eg2:tmfonethree}
We take $\Ca\lra \Ba$ to be $tmf\lra tmf_1(3)$ at the prime 2 and
$k={\mathbb{F}}_2$.  As in the discussion of the relatively Gorenstein
condition (Example \ref{eg:relGortmf}(ii))  $tmf_1(3)$ is $tmf$ extended by a finite
complex, it is small. We have also seen $tmf_1(3)$ is Gorenstein of shift
$-11$, so that by Theorem \ref{thm:THHGordescent}, the ring $THH(tmf_1(3); {\mathbb{F}}_2)$ is Gorenstein of shift
$8$. Since the complex is self-dual of dimension 12, Theorem
\ref{thm:THHGordescent} shows that   $THH(tmf; {\mathbb{F}}_2))$ is Gorenstein of
shift $20$. 

More directly, using Example \ref{eg:relGortmf} (ii)  we observe that since $tmf_1(3)$ is small over
$tmf$ and ${\mathbb{F}}_2$ is small over $tmf_1(3)$ then 
${\mathbb{F}}_2$ is small over $tmf$. Since $tmf_1(3)$ is Gorenstein
of shift $-23$ we may apply Corollary
\ref{cor:gor} to deduce $THH(tmf; {\mathbb{F}}_2))$ is Gorenstein of
shift $20$. 
\end{example}

\subsection{New examples}

The possibilities are innumerable, but we select three for
illustration. 

\begin{example}
\label{eg2:ku}
{\em (Example \ref{eg1:kukupvone} revisited again: $ku$)}
If we take $\Ca\lra \Ba$ to be $ku\lra \Z$ and $k=\Fp$,  we find $A\sim C_*(S^{3})$,
and from the fact that $THH(\Z; \Fp)$ is Gorenstein of shift $-2$ we
deduce from Theorem \ref{thm:THHGordescent} that the ring $THH(ku;\Fp)$ is Gorenstein of
shift $3+(-2)=1$. 

The spectral sequence of Lemma \ref{lem:connSS} gives a calculation. If $p$ is odd, there
can be no differentials and 
$THH(ku; \Fp)=\Fp [\mu_{2p}]\tensor \Lambda (\lambda_{2p-1},
\lambda_3)$. 

If $p=2$ the only necessary input would be  to know whether the differential
$d_{4}(\mu_{4})$ is zero or not. If it is zero then 
$THH_*(ku; \Ftwo)=\Ftwo [\mu_{4}]\tensor \Lambda (\lambda_{3},
\lambda_3')$. If it is non-zero then 
$THH_*(ku; \Ftwo)=\Ftwo [\mu_{8}]\tensor \Lambda
(\lambda_{3},\lambda_7)$. The second of these is what actually
happens, as one may see from Dundas's Lemma applied to $ko\lra ku$
together with the result for $ko$ described in Example \ref{eg2:ko}. 

Of course we get the same conclusion by working with $\Ca\lra \Ba$ to be
$ku \lra \Fp$ and $k=\Fp$ we find $A\sim C_*(S^1\times S^{3})$ so
that by Corollary \ref{cor:gor} the ring $THH(ku;\Fp)$ is Gorenstein of shift $[3+1]-3=1$. 
\end{example}

\begin{example}
\label{eg2:en}
We take $\Ca \lra \Ba$ to be $e_n\lra \Fp$, where $e_n$ is the connective Lubin-Tate commutative ring spectrum with
homotopy $W({\mathbb{F}}_{p^n})[[u_1, \ldots , u_{n-1}]][u]$. From its
homotopy we see that it is Gorenstein of shift $-n-3$. From Corollary
\ref{cor:gor} we conclude $THH(e_n; \Fp)$ is Gorenstein of
shift $n$. 
\end{example}

These examples are iterable. 
\begin{example}
\label{eg2:veen}
Veen's calculation \cite{Veen} of the 
homotopy ring of the double THH of $k$, as $k[\mu'_2, \mu''_2]\tensor
\Lambda (\lambda_3)$ gives an example whose coefficient ring is
Noetherian and of Krull dimension 2.  

If we are given map  $R\lra k$ we may take $\Ca\lra \Ba$ to be $THH(R;
k)\lra THH(k)$,  and take $A$ to be its cofibre. Now apply Dundas's
Lemma to obtain a cofibre sequence
$$A\lra THH(THH(R;k); THH(k)) \lra THH(THH(k)). $$

Unfortunately there seems to be no obvious example for which $A$ is
finite dimensional, or even Noetherian. For example, if we take $R=\Z$
it seems $\pi_*A= \Fp [\mu_2]/(\mu_2^p)\tensor \Gamma (\gamma_{2p})$. 
\end{example}

\subsection{Discussion}
The main obstacle to finding more examples is the need to ensure that the
coefficient rings should be Noetherian, which seems rather rare for
$THH$. In the cases with complete calculations and Noetherian rings,  the coefficient rings are all themselves
Gorenstein. We have not yet  found an example where the ring
spectrum is proxy regular and Gorenstein except when the coefficients are already Gorenstein. 

Our analysis is based on  the  cofibre sequence $A\lra \TbC
\lra \TbB$  and requires that it is \pifinite . We have relied on the fact
that   if $A$ is finite dimensional and $k$ is of characteristic $p$, 
the cofibre sequence is \pifinite. In this case  $\TbC$ will have the
same Krull  dimension as $\TbB$. 

\appendix

\begin{center}
{\large 
Appendices }

\end{center}
We end with two appendices describing closely related
phenomena. They do not form part of the argument, and are included
for comparison. 

 \section{Thom spectra}

Given a 3-fold loop map $f: X\lra BF$, Blumberg-Cohen-Schlichtkrull \cite{BCS} prove
$$THH(Mf)\simeq Mf \sm BX_+. $$
Since $f$ is a 3-fold loop map, $X$ is a 3-fold loop space and
$BX\simeq \Omega BBX$. We expect that $THH(Mf)=C_*(BX; Mf)$ 
being Gorenstein over $Mf$ will be related to $C^*(BBX; Mf)=F(BBX_+, Mf)$ being
Gorenstein over $Mf$. 

\begin{example} {\em ($\Fp$ revisited.)} 
This example is closely related to the fact that  $THH(\Fp)$ is Gorenstein of shift $-3$.

Mahowald  \cite{Mahowald} showed that the
Eilenberg-MacLane spectrum $\Fp$ is the Thom spectrum of 
a map $\Omega^2S^3 \lra BF$. Although this is only a double loop
space map, it is shown in \cite[1.3]{BCS}  that 
$$THH(\Fp)\simeq \Fp \sm \Omega S^3_+=:C_*(\Omega S^3),  $$
but this is not an equivalence of ring spectra, since the ring
structures in homotopy groups are different. 

Since $S^3$ is a 3-manifold,  $C^*(S^3)$ is Gorenstein of shift $-3$,
and by Morita invariance of the Gorenstein condition \cite{DGI1}, we conclude that 
$C_*(\Omega S^3)$ is Gorenstein of shift $-3$. Although it is
precisely parallel, this doesn't directly imply anything about $THH(\Fp)$. 
\end{example}

\begin{example} {\em ($\Z$ revisited.)}
This is closely related to the fact that $THH(\Z; \Z/p)$ is Gorenstein of shift $-2$. 

Mahowald \cite{Mahowald} proved that the
Eilenberg-MacLane spectrum $\Z$ is the Thom spectrum of 
a map $\Omega^2S^3\langle 3\rangle \lra BF$. Although it is only a
double loop map it is shown as \cite[1.4]{BCS}  that 
$THH(\Z)\simeq \Z  \sm \Omega S^3\langle 3 \rangle_+$, and
that  
$$THH(\Z; \Z/p)\simeq \Fp \sm \Omega S^3\langle 3
\rangle_+=C_*(\Omega S^3\langle 3\rangle).$$

Applying cochains to the fibration 
$$K(\Z, 2) \lra S^3\langle 3\rangle  \lra S^3  $$
we get a cofibre sequence of ring spectra, and a standard calculation
shows this is \pifinite . By Corollary \ref{cor:GorascentNoeth}, the cochains on the total
space is Gorenstein if the cochains on the base and the cochains on
the fibre are, and that the shifts add. Since $S^3$ is a 3-manifold, $C^*(S^3)$ is Gorenstein of shift $-3$. 
On the other hand $C_*(\Omega BU(1))=C_*(U(1))$ is Gorenstein of shift 1 since
$U(1) $ is a 1-dimensional compact Lie group, and hence by Morita 
invariance of the Gorenstein condition \cite{DGI1}, we find 
$C^*(BU(1))$ is  Gorenstein of shift 1. By Gorenstein ascent, we deduce that 
$C^*(S^3\langle 3 \rangle)$ is Gorenstein of shift
$-2=-3+1$. By Morita invariance of the Gorenstein condition \cite{DGI1} we conclude that 
$C_*(\Omega S^3\langle 3 \rangle )$ is  Gorenstein of shift $-2$.
\end{example}


\section{Dwyer-Miller and Kontsevich duality}
\label{sec:NAR1}

This section describes some known dualities for Hochschild homology and
cohomology with a similar flavour. The analogue to bear in mind is the case $R=k[x]$ on a
polynomial generator of even degree $d$. This is Gorenstein of shift
$-d-1$. Now we calculate $HH^*(R)=k[x,\alpha]$ where
$\alpha$ is of degree $-d-1$, and $HH_*(R)=k[x, \beta]$ with $\beta$
of  degree $d+1$, so we have $HH^*(R)=\Sigma^{-d-1}HH_*(R)$.

In this section we work under $k$ rather than under the sphere
spectrum $\bbS$. In particular, we assume that our rings are
$k$-algebras so 
the situation is very different to the one considered in the body of
the paper.  The discussion developed from the work of Cohen-Jones \cite{CohenJones}.

The following assumption is more often satisfied by objects like group rings
than the commutative rings we have been concerned with in the body of
the paper. 

\begin{assumption}
\label{ass:DwyerMiller}
We assume that $S=k$, so that $R^e=R\tensor_k R$ and that there is a ring map $R\lra R^e$ of 
$R^e$-modules. Finally, we assume that the bimodule $R$ is  induced from an
$R$-module: $R=R^e\tensor_R k$. 
\end{assumption}

\begin{prop} (Dwyer-Miller)
\label{prop:DM}
If Assumption \ref{ass:DwyerMiller} holds,  $k$ is small over $R$ and $R\lra k$ is Gorenstein of shift $a$ then 
$$HH^*(R;P)\cong \Sigma^a HH_*(R; P)$$ 
for all bimodules $P$. 
\end{prop}

\begin{remark}
There seems no prospect of a result like this for $THH$ since (as in
the case with $P=k$ for instance, $THH^{\bullet}(k)$ (for example) is not bounded below. 
\end{remark}

\begin{proof}
We argue as follows, where the third equivalence requires $k$ to be
small, and where $P^{ad}$ is the restriction of $P$ along the map
$R\lra R^e$. 
$$\begin{array}{rcl}
HH^*(R;P)&=&\Hom_{R^e}(R,P)\\
&=&\Hom_R(k, P^{ad})\\
&\simeq &\Hom_R(k, R)\tensor_R P^{ad}\\
&=&\Sigma^ak \tensor_R P^{ad}\\
&=&\Sigma^aR \tensor_{R^e} P\\
&=&\Sigma^aHH_*(R ;P)
\end{array}$$
\end{proof}

\begin{remark}
Take $R=C_*(\Omega X)$ for a simply connected $d$-manifold $X$. Since
$X$ can be given the structure of a finite CW complex, $C^*(X)$ is
finitely built from $k$, and applying $\Hom_{C^*(X)}(\cdot , k)$ we
see that $k$ is small over $C_*(\Omega X)$ and the hypotheses of
Proposition \ref{prop:DM} hold.  

By the Morita invariance of the Gorenstein condition \cite[Proposition
8.5]{DGI1} this is Gorenstein
of shift $a=-d$. Finally, note that   $HH_*(R)=H_*(\Lambda X)$, so
taking $P=R$, we see $\Sigma^d H_*(\Lambda X)=HH^*(R)$, showing that the shifted
homology of the free loop space has a ring structure; Malm \cite{Malm}
shows this 
corresponds to the Chas-Sullivan product.  
\end{remark}

\begin{cor} (Kontsevich duality)
If Assumption \ref{ass:DwyerMiller} holds, $k$ is small over $R$ and $R\lra k$ is Gorenstein with shift $a$
then $R^e\lra R$ is relatively Gorenstein with shift $a$. 
\end{cor}

\begin{remark}
The connection between the Gorenstein conditions on $R$ and $R^e$ in
commutative algebra is of great interest \cite{AI,
  AIL}.   
\end{remark} 

\begin{proof}
Taking $P=R^e$ in the Proposition \ref{prop:DM} we find
$$\Hom_{R^e}(R, R^e)=HH^*(R|S;R^e)=\Sigma^aHH_*(R|S;R^e)=\Sigma^a
R\tensor_{R^e}R^e=\Sigma^a R. $$
\end{proof}

Now consider the $(R|S)$-bimodule $P=k$, and note that if $R$ and $k$
are commutative rings, then so is 
$$HH_*(R|S;k)=R\tensor_{R^e}k.$$
Furthermore $R^e$ is an $R$-algebra, so we have an algebra map 
$$R\tensor_{R^e}k\lra R^e\tensor_{R^e} k \stackrel{\simeq}\lra k$$

\begin{cor} (Algebraic Ausoni-B\"okstedt duality)
If Assumption \ref{ass:DwyerMiller} holds, $k$ is small over $R$ and $R\lra k$ is Gorenstein with shift $a$
then 
$HH_*(R|S;k) \lra k$ is Gorenstein with shift $-a$. 
\end{cor}

\begin{proof}
Taking $P=k$ in the above we find
$$HH^{\bullet}(R|S;k)=\Sigma^a HH_{\bullet}(R|S;k). $$
Now we calculate
$$\begin{array}{rcl}
\Hom_{HH_{\bullet}(R|S;k)}(k,HH_{\bullet}(R|S;k))
 &=&\Sigma^{-a}\Hom_{HH_{\bullet}(R|S;k)} (k,HH^{\bullet}(R|S;k))\\
 &=&\Sigma^{-a}\Hom_{HH_{\bullet}(R|S;k)}(k, \Hom_{R^e}(R,k))\\
 &=&\Sigma^{-a}\Hom_{HH_{\bullet}(R|S;k)}(R\tensor_{R^e} k,k)\\
 &=&\Sigma^{-a}k
\end{array}$$
\end{proof}






\begin{thebibliography}{999}
\bibitem{AHL}
V.Angeltveit, M.Hill and T.Lawson 
``Topological Hochschild homology of $\ell$ and $ko$.''
American J Math {\bf 132} (2010) 297-330
\bibitem{AR}
V.Angeltveit and J. Rognes
``Hopf algebra structure on topological Hochschild homology.''
Algebraic and Geometric Topology {\bf 5} (2005) 1223-1290
\bibitem{Ausoni}
C.Ausoni
``Topological Hochschild homology of connective complex $K$-theory.''
Amer. J. Math. 127 (2005), no. 6, 1261–1313.
\bibitem{AR}
C. Ausoni and J. Rognes,
``Algebraic K-theory of topological K-theory'', 
Acta Mathematica {\bf 188} (2002), 1-39.
\bibitem{AI} L.Avramov and S.B.Iyengar
``Gorenstein algebras and Hochschild cohomology''
Michigan Math. J. {\bf 57} (2008) 17-35
\bibitem{AIL} L.Avramov, S.B.Iyengar and J.Lipman
``Reflexivity and rigidity for complexes, I''
arXiv 0904.4695
\bibitem{Bass}
H. Bass ``On the ubiquity of Gorenstein rings.'' Math. Z. {\bf 82} (1963) 8–28.
\bibitem{BCS}
A. Blumberg, R.Cohen and C.Schlichtkrull
``The topological Hochschild homology of Thom spectra and the free
loop space.''
Geom Topol {\bf 14} (2010) 1165-1242
\bibitem{Bok}
M. B\"okstedt, 
``Topological Hochschild homology of $F_p$, and $\Z$'', Preprint,
Bielefeld University.
\bibitem{CohenJones}
R.Cohen and J.D.S.Jones
``A homotopy theoretic realization of string topology.''
 Math. Ann. 324 (2002), no. 4, 773-798.
\bibitem{DGI1}
 W.G.Dwyer, J.P.C.Greenlees and S.B.Iyengar
  ``Duality in algebra and topology''
Advances in Maths. {\bf 200} (2006) 357-402
\bibitem{ringlct}
   J.P.C.Greenlees and G.Lyubeznik 
   ``Rings with a local cohomology theorem and applications to cohomology 
   rings of groups.''
   J. Pure and Applied Algebra {\bf 149} (2000) 267-285.
\bibitem{HillLawson}
M.Hill and T.Lawson 
``Topological modular forms with level structure''
arXiv: 1312.7394, 53pp. 
\bibitem{HM} M.J.Hopkins and M.E.Mahowald
``From elliptic curves to homotopy theory''
AMS Mathematical Surveys and Mongraphs, {\bf 201} (2015) 261-285
\bibitem{LawsonNaumann1}
T.Lawson and N.Naumann
Commutativity conditions for truncated Brown-Peterson spectra of height 2. (English summary)
J. Topol. {\bf 5} (2012), no. 1, 137-168. 
\bibitem{LawsonNaumann2}
T.Lawson and N.Naumann
``Strictly commutative realizations of diagrams over the Steenrod
algebra and topological modular forms at the prime 2.''
 Int. Math. Res. Not. IMRN 2014, no. 10, 2773-2813.
\bibitem{LM} A.Lindenstrauss and I. Madsen ``Topological Hochschild
homology of number rings.''
Trans AMS {\bf 352} (2000) 2179-2204
\bibitem{Mahowald} M.E.Mahowald
 ``Ring spectra which are Thom complexes.'' Duke Math. J., {\bf 46}
 (1979) 549–559
\bibitem{MS} J.E.McClure and R.E.Staffeldt
``On the topological Hochschild homology of $bu$ I''
AJM {\bf 118} (1993) 1005-1012 
\bibitem{Malm} E.J.Malm
`` String topology and the based loop space''
Preprint (2011) arXiv:1103.6198 
\bibitem{AM} Akhil Matthew
``The homology of tmf'' 
Preprint (2013) arXiv:1305.6100 

\bibitem{Veen} T.Veen
``Detecting Periodic Elements in Higher Topological Hochschild Homology
Preprint (2014) arXiv:1312.5699

\end{thebibliography}
\end{document}